\documentclass{amsart}%
\usepackage{amsmath}
\usepackage{amsfonts,color}
\usepackage{amssymb}
\usepackage{graphicx}
\usepackage{amsfonts}%
\thanks{DC and DG partially supported by NSF grant DMS 0748283.}
\newtheorem{theorem}{Theorem}
\newtheorem*{acknowledgement}{Acknowledgement}

\newtheorem{corollary}[theorem]{Corollary}

\newtheorem{definition}[theorem]{Definition}

\newtheorem{lemma}[theorem]{Lemma}

\newtheorem{problem}{Problem}
\newtheorem{proposition}[theorem]{Proposition}
\newtheorem{remark}[theorem]{Remark}

\numberwithin{equation}{section}
\numberwithin{theorem}{section}
\begin{document}
\title{Regge's Einstein-Hilbert functional on the double tetrahedron}
\author{Daniel Champion}
\address[Daniel Champion]{University of Arizona, Tucson AZ, 85721}
\email{champion@math.arizona.edu}
\author{David Glickenstein}
\address[David Glickenstein]{University of Arizona, Tucson AZ, 85721}
\email{glickenstein@math.arizona.edu}
\author{Andrea Young}
\address[Andrea Young]{University of Arizona, Tucson AZ, 85721}
\email{ayoung@math.arizona.edu}

\begin{abstract}
The double tetrahedron is the triangulation of the three-sphere gotten by
gluing together two congruent tetrahedra along their boundaries. As a
piecewise flat manifold, its geometry is determined by its six edge lengths,
giving a notion of a metric on the double tetrahedron. We study notions of
Einstein metrics, constant scalar curvature metrics, and the Yamabe problem on
the double tetrahedron, with some reference to the possibilities on a general
piecewise flat manifold. The main tool is analysis of Regge's Einstein-Hilbert
functional, a piecewise flat analogue of the Einstein-Hilbert (or total scalar
curvature) functional on Riemannian manifolds. We study the
Einstein-Hilbert-Regge functional on the space of metrics and on discrete
conformal classes of metrics.

\end{abstract}
\maketitle

\section{Introduction}

\renewcommand{\arraystretch}{1.5}It is well-known that Ricci-flat metrics on
closed Riemannian manifolds of dimension at least three are critical points of
the Einstein-Hilbert functional%
\[
\mathcal{EH}\left(  M,g\right)  =\int_{M}R_{g}dV_{g},
\]
where $R_{g}$ and $dV_{g}$ are the scalar curvature and volume form for the
Riemannian manifold $\left(  M,g\right)  .$ Since there are topological
restrictions to being Ricci-flat (e.g., the Cheeger-Gromoll splitting theorem
\cite{CG}), one may restrict to the subset of Riemannian manifolds with volume
equal to $1$ so that critical points of the constrained problem are Einstein
manifolds. Equivalently, one can consider a normalized Einstein-Hilbert
functional
\[
\mathcal{NEH}\left(  M^{n},g\right)  =\frac{\int_{M}R_{g}dV_{g}}{\left(
\int_{M}dV_{g}\right)  ^{\left(  n-2\right)  /n}},
\]
whose critical points are Einstein manifolds. Einstein manifolds are of
interest because the Einstein metric is, in some sense, a most symmetric or
\textquotedblleft best\textquotedblright\ geometry for the manifold. In trying
to prove a classification theorem such as Thurston's geometrization
conjecture, one may try to find a best geometry by trying to optimize a
geometric functional such as $\mathcal{EH}$ and by studying both convergence
and degenerations to try to capture all possible \textquotedblleft
best\textquotedblright\ geometries (see \cite{And}).

Related to the study of $\mathcal{EH}$ and Einstein manifolds is the
well-known Yamabe problem, which asks whether one can find a constant scalar
curvature metric within a conformal class or, equivalently, if one can find a
critical point for $\mathcal{NEH}$ restricted to a conformal class. The
solution was completed by R. Schoen, based on important contributions from
Yamabe, Trudinger, and Aubin (see \cite{LP} for an overview and complete proof).

In assigning geometry to a topological manifold, an alternative to the
Riemannian approach is that of piecewise flat geometry. A piecewise flat
manifold is a triangulation together with edge lengths that determine a
Euclidean geometry on each simplex in the triangulation. In 1961, T. Regge
\cite{Reg} suggested a functional defined on piecewise flat manifolds which is
analogous to $\mathcal{EH}$. We call this functional the
Einstein-Hilbert-Regge functional and denote it as $\mathcal{EHR}$. Study of
this functional as an action for general relativity has led to a wide array of
work on Regge calculus and lattice gravity (for a survey, see \cite{Ham}).\ It
was later shown that $\mathcal{EHR}$ and $\mathcal{EH}$ are related in the
sense that appropriately finer piecewise flat triangulations which converge to
a Riemannian manifold lead to convergence of the functionals. In fact, it was
proven that the associated curvature measures converge \cite{CMS}. Thus
$\mathcal{EHR}$ is a discretization of $\mathcal{EH}$, and could potentially
be used to approximate $\mathcal{EH}$. Such an approach is an alternative to
discretizing the Einstein equations themselves. By discretizing the functional
instead of its Euler-Lagrange equation, we hope to produce an approximation of
the Euler-Lagrange equation whose behavior mimics that of the smooth case.
This approach has been applied in a number of contexts, such as computer
graphics, computational mechanics, and computational dynamics, and it is the
main focus of the fields of discrete differential geometry and discrete
exterior calculus (see, e.g., \cite{BSu}, \cite{DHM}, \cite{DP}, \cite{MDSB}).

In addition, we can use a definition of conformal class in \cite{G3} to
formulate a discrete version of the Yamabe problem. However, this does not
allow us to reformulate the functional in the same way as in the smooth
setting, which allows $\mathcal{NEH}$ to be rewritten in a relatively simply
way as function of the conformal factor. Instead, we are forced to work
entirely with variation formulas for curvature.

The purpose of this paper is to consider the Einstein-Hilbert-Regge functional
on the simplest possible triangulation of a three-manifold without boundary:
the double tetrahedron. Even on this small triangulation, the behavior of the
$\mathcal{EHR}$ functional is rich and complex. In particular, on the double
tetrahedron we do not have a complete answer to the uniqueness of Einstein
metrics, a complete understanding of the Yamabe problem, or a calculation of
the Yamabe invariant.

In \S \ref{background section}, we give some background on piecewise flat
manifolds, and we define the $\mathcal{EHR}$ functional and two normalized
versions of it. Additionally, we describe the geometry of the double
tetrahedron, and we set the notation for the remainder of the paper. In
\S \ref{metric section}, we consider length variations of piecewise flat
metrics on the double tetrahedron. Critical points of the normalized
$\mathcal{EHR}$ functionals are geometrically significant and yield
definitions of Einstein metrics in the piecewise flat setting. We study the
convexity of the functionals at these points. In \S \ref{conformal section},
we discuss discrete conformal variations of a piecewise flat metric described
in \cite{G3} (following \cite{Luo1}, \cite{RW}, \cite{SSP}). The critical
points of the normalized $\mathcal{EHR}$ functionals with respect to a
conformal variation give rise to a notion of constant scalar curvature
piecewise flat metrics. On the double tetrahedron, we are able to provide a
partial classification of such metrics and are able to show existence in every
conformal class. Additionally, we study the convexity of the curvature
functionals at Einstein metrics. Finally, in \S \ref{general section} we
discuss the Yamabe invariant on both the double tetrahedron and on general
piecewise flat manifolds.

\begin{acknowledgement}
We would like to thank the other members of the GEOCAM project, especially
Alex Henniges, for many conversations on these topics. We would also like to
acknowledge helpful conversations with Jeff Weeks and Tom Banchoff.
\end{acknowledgement}

\section{Background and notation\label{background section}}

In this section we will secure notation for the rest of the paper. Most of the
notation follows \cite{G3}. We will also provide the necessary background on
piecewise flat manifolds, and we will define the double tetrahedron.

\subsection{Geometry of the tetrahedron\label{section tetra geometry}}

Consider a Euclidean tetrahedron determined by four vertices numbered
$1,2,3,4.$ The tetrahedron has six edge lengths, and we denote the length of
the edge between vertices $i$ and $j$ by $\ell_{ij}.$ Since edge lengths arise
from a nondegenerate tetrahedron, they satisfy a particular condition.

\begin{definition}
Consider the matrix $A:$%
\[
A=%
\begin{bmatrix}
0 & 1 & 1 & 1 & 1\\
1 & 0 & \ell_{12}^{2} & \ell_{13}^{2} & \ell_{14}^{2}\\
1 & \ell_{12}^{2} & 0 & \ell_{23}^{2} & \ell_{24}^{2}\\
1 & \ell_{13}^{2} & \ell_{23}^{2} & 0 & \ell_{34}^{2}\\
1 & \ell_{14}^{2} & \ell_{24}^{2} & \ell_{34}^{2} & 0
\end{bmatrix}
.
\]
Let $\operatorname{CM}_{3}=\det\left(  A\right)  $. \ 
\end{definition}

The quantity $\operatorname{CM}_{3}$ is related to the volume of the
tetrahedron:%
\[
Vol=\sqrt{\frac{CM_{3}}{288}}.
\]
The quantity $\operatorname{CM}_{3}$ is a special case of a Cayley-Menger
determinant $\operatorname{CM}_{n}$, which determines the volume of an
$n$-simplex in a similar way. If $\operatorname{CM}_{3}\leq0$, then there is
no nondegenerate Euclidean tetrahedron with those edge lengths. Notice that
$\operatorname{CM}_{3}$ is a polynomial of degree six in the edge lengths.

The angles of a Euclidean triangle are determined by the edge lengths via the
cosine law. The dihedral angles of a tetrahedron can then be calculated from
the angles at the faces using the spherical cosine law. We use $\beta_{e}$ to
refer to the dihedral angle of a tetrahedron at edge $e.$ If we wish to
emphasize that it is in tetrahedron $t,$ we denote it as $\beta_{e<t}.$ In the
sequel, $\tau<\sigma$ or $\sigma>\tau$ will mean that $\tau$ is a sub-simplex
of $\sigma.$

\subsection{Piecewise flat manifolds}

In this section we recall some definitions related to piecewise flat manifolds.

The double tetrahedron is a particular case of a \emph{triangulated piecewise
flat manifold}. By a triangulation, we mean a collection of simplices
identified along sub-simplices. Note that the triangulation need not be a
simplicial complex (for instance, in the double tetrahedron there are two
tetrahedra associated to the same collection of vertices). The dimension of a
triangulation is that of its highest dimensional simplex. A three-dimensional
triangulation $\mathcal{T}=\left(  V,E,F,T\right)  $ has a collection of
vertices (denoted $V$), edges (denoted $E$), faces (denoted $F$), and
tetrahedra (denoted $T$). A triangulated piecewise flat manifold is denoted as
$\left(  M,\mathcal{T},\ell\right)  ,$ where $M$ is a manifold, $\mathcal{T}$
is a triangulation of $M,$ and $\ell$ is a metric according to the following definition.

\begin{definition}
\label{met def gen}A vector $\ell\in\mathbb{R}^{\left\vert E\right\vert }$
such that each simplex can be realized as a Euclidean simplex with edge
lengths determined by $\ell$ is called a \emph{metric} for the triangulated
manifold $\left(  M,\mathcal{T}\right)  ,$ and $\left(  M,\mathcal{T}%
,\ell\right)  $ is called a \emph{triangulated piecewise flat manifold}. The
space of all metrics is denoted $\mathfrak{met}\left(  M,\mathcal{T}\right)
.$
\end{definition}

Note that the condition for a metric can be described using Cayley-Menger
determinants of the type described in Section \ref{section tetra geometry}.

We will restrict to the case that $M$ is three-dimensional. There are several
quantities associated to $\left(  M,\mathcal{T},\ell\right)  $:

\begin{definition}
The \emph{edge curvature} of an edge $e$ is
\begin{equation}
K_{e}=\left(  2\pi-\sum_{t\in T}\beta_{e<t}\right)  \ell_{e},
\end{equation}
where $\beta_{e<t}$ is the dihedral angle at edge $e$ in tetrahedron $t$, and
$\ell_{e}$ is the edge length.
\end{definition}

Now we can define some functionals on piecewise flat manifolds.

\begin{definition}
The \emph{total length} of $\left(  M,\mathcal{T},\ell\right)  $ is
\begin{equation}
\mathcal{L}\left(  M,\mathcal{T},\ell\right)  =\sum_{e\in E}\ell_{e}.
\end{equation}
Let $V_{t}$ be the volume of tetrahedron $t.$ Then the \emph{volume} of
$\left(  M,\mathcal{T},\ell\right)  $ is
\begin{equation}
\mathcal{V}\left(  M,\mathcal{T},\ell\right)  =\sum_{t\in T}V_{t}.
\end{equation}
The \emph{Einstein-Hilbert-Regge functional} is
\begin{equation}
\mathcal{EHR}\left(  M,\mathcal{T},\ell\right)  =\sum_{e\in E}K_{e}.
\end{equation}

\end{definition}

We will also consider two normalizations of the $\mathcal{EHR}$ functional.
Volume normalization is quite natural and is the usual normalization
considered in the Riemannian setting (see \S \ref{general section}). However,
since the formula for volume of a simplex is quite complicated, one may also
consider a normalization which is linear in the edge lengths.

\begin{definition}
The \emph{length normalized Einstein-Hilbert-Regge functional} is
\begin{equation}
\mathcal{LEHR}\left(  M,\mathcal{T},\ell\right)  =\frac{\mathcal{EHR}\left(
M,\mathcal{T},\ell\right)  }{\mathcal{L}\left(  M,\mathcal{T},\ell\right)  }.
\label{LEHR}%
\end{equation}
The \emph{volume normalized Einstein-Hilbert-Regge functional} is
\begin{equation}
\mathcal{VEHR}\left(  M,\mathcal{T},\ell\right)  =\frac{\mathcal{EHR}\left(
M,\mathcal{T},\ell\right)  }{\mathcal{V}\left(  M,\mathcal{T},\ell\right)
^{1/3}}. \label{VEHR}%
\end{equation}

\end{definition}

The normalizations are defined so that the functionals take the same value if
all lengths are scaled by the same positive constant.

Following \cite{G3}, we can now define what we consider to be Einstein
metrics, which will depend on the normalization. We call these the Einstein
metrics because they are the critical points of the corresponding normalized functionals.

\begin{definition}
The metric $\ell\in\mathfrak{met}\left(  M,\mathcal{T}\right)  $ is an
\emph{$\mathcal{L}$-Einstein metric} if there exists $\lambda_{\mathcal{L}}%
\in\mathbb{R}$ such that for all $e\in E$,
\[
K_{e}=\lambda_{\mathcal{L}}\ell_{e}.
\]
Here $\lambda_{\mathcal{L}}=\mathcal{LEHR}$. The metric $\ell\in
\mathfrak{met}\left(  M,\mathcal{T}\right)  $ is a \emph{$\mathcal{V}%
$-Einstein metric} if there exists $\lambda_{\mathcal{V}}\in\mathbb{R}$ such
that for all $e\in E$,
\[
K_{e}=\lambda_{\mathcal{V}}V_{e}.
\]
Here $\lambda_{\mathcal{V}}=\frac{\mathcal{EHR}}{3\mathcal{V}}$, and
$V_{e}=\ell_{e}\frac{\partial\mathcal{V}}{\partial\ell_{e}}.$
\end{definition}

\begin{remark}
Note that $\mathcal{L}$-Einstein metrics satisfy
\[
\frac{\partial\mathcal{EHR}}{\partial\ell_{e}}=\lambda_{\mathcal{L}}%
\frac{\partial\mathcal{L}}{\partial\ell_{e}},
\]
and $\mathcal{V}$-Einstein metrics satisfy
\[
\frac{\partial\mathcal{EHR}}{\partial\ell_{e}}=\lambda_{\mathcal{V}}%
\frac{\partial\mathcal{V}}{\partial\ell_{e}}.
\]

\end{remark}

\subsection{Geometry and topology of the double tetrahedron.}

Most of this paper will be concerned with the double tetrahedron.

\begin{definition}
\label{double tetra triang}The \emph{double tetrahedron} $\operatorname{DT}$
is the triangulation of the three-sphere obtained by identifying the
corresponding boundary faces of two disjoint tetrahedra.
\end{definition}

One should consider the double tetrahedron as $\operatorname{DT}=\left(
S^{3},\mathcal{T}\right)  ,$ where $\mathcal{T}$ is the triangulation
described in Definition \ref{double tetra triang}. The double tetrahedron is
given a geometry by specifying the edge lengths (or metric), and these six
edge lengths determine the two Euclidean tetrahedra which make up the
manifold. Note that the two tetrahedra are necessarily congruent, which leads
to the following definition.

\begin{definition}
A single tetrahedron in $\operatorname{DT}$ is called the \emph{generating
tetrahedron}.
\end{definition}

The set of metrics on the double tetrahedron can be described succinctly as follows.

\begin{definition}
\label{met def}The space of metrics on the double tetrahedron is the set:%
\[
\mathfrak{met}\left(  \operatorname{DT}\right)  =\left\{  \vec{\ell}%
\in\mathbb{R}^{6}:\operatorname{CM}_{3}>0\right\}  .
\]

\end{definition}

Definition \ref{met def} gives the same set as Definition \ref{met def gen},
but describe it more explicitly. We will often use the term double tetrahedron
to refer to the double tetrahedron with an arbitrary metric $\ell$. The edge
curvatures of the double tetrahedron can be expressed succinctly as%
\begin{equation}
K_{e}=\left(  2\pi-2\beta_{e}\right)  \ell_{e}, \label{edge curv dt}%
\end{equation}
where $\beta_{e}$ is the dihedral angle in the generating tetrahedron at edge
$e.$ Note the following important property of the double tetrahedron.

\begin{lemma}
\label{edge curv pos lemma}On the double tetrahedron, $K_{e}>0$ for any metric
$\ell.$
\end{lemma}

\begin{proof}
Since in a (nondegenerate) tetrahedron, each dihedral angle is less than
$\pi,$ the lemma follows from formula (\ref{edge curv dt}).
\end{proof}

We will label the vertices $\left\{  1,2,3,4\right\}  $ and edges will be
denoted as $ij,$ where $i,j\in\left\{  1,2,3,4\right\}  .$ For instance, in
regard to the edge $12$ between vertices $1$ and $2,$ we will refer to the
length of the edge as $\ell_{12},$ the dihedral angle at the edge (in the
generating tetrahedron) as $\beta_{12},$ and the edge curvature as $K_{12}.$

\section{Metric variations\label{metric section}}

In this section we will study two normalizations of the Einstein-Hilbert-Regge
functional on the double tetrahedron with the primary goal of finding Einstein
metrics. To do so, we need to define the following subspaces of
$\mathfrak{met}(\operatorname{DT})$.

\begin{definition}
The space of \emph{length normalized edge lengths} on the double tetrahedron
is the set:%
\[
\mathfrak{met}_{\mathcal{L}}\left(  \operatorname{DT}\right)  =\left\{
\vec{\ell}\in\mathfrak{met}\left(  \operatorname{DT}\right)  :\sum
\limits_{\left\{  i,j\right\}  \in E}\ell_{ij}=1\right\}  .
\]

The space of \emph{volume normalized edge lengths} on the double tetrahedron
is the set:%
\[
\mathfrak{met}_{\mathcal{V}}\left(  \operatorname{DT}\right)  =\left\{
\vec{\ell}\in\mathfrak{met}\left(  \operatorname{DT}\right)  :\mathcal{V}%
=1\right\}  .
\]

\end{definition}

The space $\mathfrak{met}$ is defined with an open condition, and hence has
the structure of an open six-dimensional manifold. Note that $1$ is a regular
value of the functions $\mathcal{L}$ and $\mathcal{V}$, and hence
$\mathfrak{met}_{\mathcal{L}}$ and $\mathfrak{met}_{\mathcal{V}}$ have the
structures of smooth five-dimensional submanifolds of $\mathfrak{met}$.

In this section, will analyze the variational properties of the
$\mathcal{LEHR}$ and $\mathcal{VEHR}$ functionals. To this end, we require the
following variational results which follow from the Schl\"{a}fli formula (see,
e.g., \cite{G3}).

\begin{proposition}
\label{length derivative of EHR PROP}For the double tetrahedron
$\operatorname{DT},$%

\begin{subequations}
\begin{align}
&  \frac{\partial\mathcal{EHR}\left(  \operatorname{DT},\ell\right)
}{\partial\ell_{ij}}=2\pi-2\beta_{ij}=\frac{K_{ij}}{\ell_{ij}},\\
&  \frac{\partial\mathcal{LEHR}\left(  \operatorname{DT},\ell\right)
}{\partial\ell_{ij}}=\mathcal{L}^{-1}\left(  \frac{K_{ij}}{\ell_{ij}}%
-\frac{\mathcal{EHR}\left(  M,T,\ell\right)  }{\mathcal{L}\left(
M,\mathcal{T}\ell\right)  }\right)  ,\\
&  \frac{\partial\mathcal{VEHR}\left(  \operatorname{DT},\ell\right)
}{\partial\ell_{ij}}=\mathcal{V}^{-\frac{1}{3}}\left(  \frac{K_{ij}}{\ell
_{ij}}-\frac{\mathcal{EHR}\left(  M,\mathcal{T}\ell\right)  }{3\mathcal{V}%
\left(  M,T,\ell\right)  }\frac{\partial\mathcal{V}}{\partial\ell_{ij}%
}\right)  .
\end{align}

\end{subequations}
\end{proposition}

One may guess that the double tetrahedron with all lengths equal is somehow
special. We call this the \emph{equal length metric} and note that it is
unique up to scaling. Our main results in this section are the following:

\begin{theorem}
\label{Combined eigen decomposition THM}

\begin{enumerate}
\item On the double tetrahedron, equal length metrics are Einstein metrics
with respect to both $\mathcal{LEHR}$ and $\mathcal{VEHR}$.

\item The eigenspaces and eigenvalues for the Hessian matrices of
$\mathcal{LEHR}$ and $\mathcal{VEHR}$ at equal length metrics (with edge
lengths $k$) are the following:%
\[
{\footnotesize
\begin{tabular}
[c]{|l|l|l|}\hline
eigenspace & spanning vectors & eigenvalues\\\hline
$V_{\lambda_{1}}$ & \multicolumn{1}{|c|}{$%
\begin{array}
[c]{c}%
\left(  1,0,0,0,0,-1\right) \\
\left(  0,1,0,0,-1,0\right) \\
\left(  0,0,1,-1,0,0\right)
\end{array}
$} & \multicolumn{1}{|r|}{$%
\begin{array}
[c]{l}%
\lambda_{1}^{\mathcal{LEHR}}=\frac{2\sqrt{2}}{9}k^{-2}\approx0.313\cdot
k^{-2},\\
\lambda_{1}^{\mathcal{VEHR}}=2^{\frac{7}{6}}3^{-\frac{2}{3}}k^{-2}\left(
2^{\frac{3}{2}}+9\pi-9\arccos\left(  \frac{1}{3}\right)  \right) \\
\;\;\;\;\;\;\;\;\;\approx21.611\cdot k^{-2},
\end{array}
$}\\\hline
$V_{\lambda_{2}}$ & \multicolumn{1}{|c|}{$%
\begin{array}
[c]{c}%
\left(  0,1,-1,-1,1,0\right) \\
\left(  1,-\frac{1}{2},-\frac{1}{2},-\frac{1}{2},-\frac{1}{2},1\right)
\end{array}
$} & $%
\begin{array}
[c]{c}%
\multicolumn{1}{l}{\lambda_{2}^{\mathcal{LEHR}}=-\frac{2\sqrt{2}}{3}%
k^{-2}\approx-0.94\cdot k^{-2},}\\
\lambda_{2}^{\mathcal{VEHR}}=2^{\frac{7}{6}}3^{\frac{1}{3}}k^{-2}\left(
7\pi-2^{\frac{3}{2}}-7\arccos\left(  \frac{1}{3}\right)  \right) \\
\multicolumn{1}{l}{\;\;\;\;\;\;\;\;\;\approx34.145\cdot k^{-2},}%
\end{array}
$\\\hline
$V_{0}$ & \multicolumn{1}{|c|}{$%
\begin{array}
[c]{c}%
\left(  1,1,1,1,1,1\right)
\end{array}
$} & $%
\begin{array}
[c]{c}%
0.
\end{array}
$\\\hline
\end{tabular}
\ \ \ \ \ \ }%
\]

\item Equal length metrics are the only critical points of the $\mathcal{LEHR}%
$ functional and hence are the only $\mathcal{L}$-Einstein metrics. These
metrics are saddle points.

\item Equal length metrics are local minima of the $\mathcal{VEHR}$ functional.
\end{enumerate}
\end{theorem}

The proof of Theorem~\ref{Combined eigen decomposition THM}.1 follows directly
from the formulas presented in Proposition~\ref{length derivative of EHR PROP}.

One can compute the eigenvalues and eigenvectors of the Hessian matrix for
both the $\mathcal{LEHR}$ and $\mathcal{VEHR}$ functionals at equal length
metrics to obtain the decomposition presented in
Theorem~\ref{Combined eigen decomposition THM}.2. In fact, we prove a more
general lemma which will be used later.

\begin{lemma}
\label{hessian lemma}Let $\mathcal{F}$ be either $\mathcal{LEHR}$ or
$\mathcal{VEHR}$. For the length variation:%
\begin{align}
\ell\left(  t\right)   &  =\left(  \ell_{12}\left(  t\right)  ,\ell
_{13}\left(  t\right)  ,\ell_{14}\left(  t\right)  ,\ell_{23}\left(  t\right)
,\ell_{24}\left(  t\right)  ,\ell_{34}\left(  t\right)  \right)
,\label{length variation}\\
&  =\left(  t,1,1,1,1,t\right)  ,\nonumber
\end{align}
define the following quantities%
\[%
\begin{array}
[c]{ccc}%
a_{\mathcal{F}}\left(  t\right)  =\left.  \frac{\partial^{2}\mathcal{F}%
}{\partial\ell_{12}\partial\ell_{13}}\right\vert _{\ell\left(  t\right)  }, &
b_{\mathcal{F}}\left(  t\right)  =\left.  \frac{\partial^{2}\mathcal{F}%
}{\partial\ell_{12}^{2}}\right\vert _{\ell\left(  t\right)  }, &
c_{\mathcal{F}}\left(  t\right)  =\left.  \frac{\partial^{2}\mathcal{F}%
}{\partial\ell_{13}\partial\ell_{14}}\right\vert _{\ell\left(  t\right)  },\\
d_{\mathcal{F}}\left(  t\right)  =\left.  \frac{\partial^{2}\mathcal{F}%
}{\partial\ell_{13}^{2}}\right\vert _{\ell\left(  t\right)  }, &
e_{\mathcal{F}}\left(  t\right)  =\left.  \frac{\partial^{2}\mathcal{F}%
}{\partial\ell_{12}\partial\ell_{34}}\right\vert _{\ell\left(  t\right)  }, &
f_{\mathcal{F}}\left(  t\right)  =\left.  \frac{\partial^{2}\mathcal{F}%
}{\partial\ell_{13}\partial\ell_{24}}\right\vert _{\ell\left(  t\right)  }.
\end{array}
\]
Then the eigenspaces and eigenvalues of the Hessian of $\mathcal{F}$ are
\[%
\begin{tabular}
[c]{|l|c|l|}\hline
eigenspace & spanning vectors & eigenvalues\\\hline
$V_{\lambda_{1}}$ & $%
\begin{array}
[c]{c}%
\left(  1,0,0,0,0,-1\right)
\end{array}
$ & \multicolumn{1}{|c|}{$b_{\mathcal{F}}-e_{\mathcal{F}}$}\\\hline
$V_{\lambda_{2}}$ & $%
\begin{array}
[c]{c}%
\left(  0,1,0,0,-1,0\right) \\
\left(  0,0,1,-1,0,0\right)
\end{array}
$ & \multicolumn{1}{|c|}{$d_{\mathcal{F}}-f_{\mathcal{F}}$}\\\hline
$V_{\lambda_{3}}$ & $%
\begin{array}
[c]{c}%
\left(  0,1,-1,-1,1,0\right)
\end{array}
$ & \multicolumn{1}{|c|}{$-4c_{\mathcal{F}}-2ta_{\mathcal{F}}$}\\\hline
$V_{\lambda_{4}}$ & $%
\begin{array}
[c]{c}%
\left(  \frac{1}{t},-\frac{1}{2},-\frac{1}{2},-\frac{1}{2},-\frac{1}{2}%
,\frac{1}{t}\right)
\end{array}
$ & \multicolumn{1}{|c|}{$-\left(  2t+\frac{4}{t}\right)  a_{\mathcal{F}}$%
}\\\hline
$V_{0}$ & $%
\begin{array}
[c]{c}%
\left(  t,1,1,1,1,t\right)
\end{array}
$ & \multicolumn{1}{|c|}{$0$}\\\hline
\end{tabular}
\ \hspace{1mm}.
\]

\end{lemma}

\begin{proof}
We see that the Hessian matrix is
\[
\left.  Hess\left(  \mathcal{F}\right)  \right\vert _{\ell\left(  t\right)  }=%
\begin{bmatrix}
b_{\mathcal{F}} & a_{\mathcal{F}} & a_{\mathcal{F}} & a_{\mathcal{F}} &
a_{\mathcal{F}} & e_{\mathcal{F}}\\
a_{\mathcal{F}} & d_{\mathcal{F}} & c_{\mathcal{F}} & c_{\mathcal{F}} &
f_{\mathcal{F}} & a_{\mathcal{F}}\\
a_{\mathcal{F}} & c_{\mathcal{F}} & d_{\mathcal{F}} & f_{\mathcal{F}} &
c_{\mathcal{F}} & a_{\mathcal{F}}\\
a_{\mathcal{F}} & c_{\mathcal{F}} & f_{\mathcal{F}} & d_{\mathcal{F}} &
c_{\mathcal{F}} & a_{\mathcal{F}}\\
a_{\mathcal{F}} & f_{\mathcal{F}} & c_{\mathcal{F}} & c_{\mathcal{F}} &
d_{\mathcal{F}} & a_{\mathcal{F}}\\
e_{\mathcal{F}} & a_{\mathcal{F}} & a_{\mathcal{F}} & a_{\mathcal{F}} &
a_{\mathcal{F}} & b_{\mathcal{F}}%
\end{bmatrix}
.
\]
Note that for each choice of $\mathcal{F}$, we have $\mathcal{F}\left(
c\hspace{0.5mm}\ell\left(  t\right)  \right)  =\mathcal{F}\left(  \ell\left(
t\right)  \right)  $ for any scalar $c>0.$ Thus $\ell\left(  t\right)  $ is in
the nullspace of the Hessian of $\mathcal{F}$. This implies the following two
equalities:%
\begin{align*}
4a_{\mathcal{F}}+te_{\mathcal{F}}+tb_{\mathcal{F}}  &  =0,\\
2c_{\mathcal{F}}+d_{\mathcal{F}}+f_{\mathcal{F}}+2ta_{\mathcal{F}}  &  =0.
\end{align*}
One can then check the vectors in the statement of the lemma to confirm that
they are eigenvectors with the corresponding eigenvalues.
\end{proof}

\begin{proof}
[Proof of Theorem~\ref{Combined eigen decomposition THM}.2.]Let $\ell_{k}$
denote the equal length metric with all lengths equal to $k.$With the added
symmetry of this length structure we get the following equalities:%
\[%
\begin{array}
[c]{c}%
a_{\mathcal{F}}=\left.  \frac{\partial^{2}\mathcal{F}}{\partial\ell
_{12}\partial\ell_{13}}\right\vert _{\ell_{k}}=\left.  \frac{\partial
^{2}\mathcal{F}}{\partial\ell_{13}\partial\ell_{14}}\right\vert _{\ell_{k}%
}=c_{\mathcal{F}},\\
b_{\mathcal{F}}=\left.  \frac{\partial^{2}\mathcal{F}}{\partial\ell_{12}^{2}%
}\right\vert _{\ell_{k}}=\left.  \frac{\partial^{2}\mathcal{F}}{\partial
\ell_{13}^{2}}\right\vert _{\ell_{k}}=d_{\mathcal{F}},\\
e_{\mathcal{F}}=\left.  \frac{\partial^{2}\mathcal{F}}{\partial\ell
_{12}\partial\ell_{34}}\right\vert _{\ell_{k}}=\left.  \frac{\partial
^{2}\mathcal{F}}{\partial\ell_{13}\partial\ell_{24}}\right\vert _{\ell_{k}%
}=f_{\mathcal{F}}.
\end{array}
\]
By Lemma \ref{hessian lemma}, the eigenspaces and eigenvalues are:%
\[%
\begin{tabular}
[c]{|l|c|c|}\hline
eigenspace & spanning vectors & eigenvalues\\\hline
$V_{\lambda_{1}}$ & $%
\begin{array}
[c]{c}%
\left(  1,0,0,0,0,-1\right) \\
\left(  0,1,0,0,-1,0\right) \\
\left(  0,0,1,-1,0,0\right)
\end{array}
$ & $b_{\mathcal{F}}-e_{\mathcal{F}},$\\\hline
$V_{\lambda_{2}}$ & $%
\begin{array}
[c]{c}%
\left(  0,1,-1,-1,1,0\right) \\
\left(  1,-\frac{1}{2},-\frac{1}{2},-\frac{1}{2},-\frac{1}{2},1\right)
\end{array}
$ & $-6a_{\mathcal{F}},$\\\hline
$V_{0}$ & $%
\begin{array}
[c]{c}%
\left(  1,1,1,1,1,1\right)
\end{array}
$ & $0.$\\\hline
\end{tabular}
\ \ \ \
\]
A calculation yields the following:%
\[%
\begin{array}
[c]{cl}%
a_{\mathcal{LEHR}}=\frac{\sqrt{2}}{9k^{2}}, & a_{\mathcal{VEHR}}=2^{\frac
{1}{6}}3^{\frac{-2}{3}}k^{-2}\left(  2\sqrt{2}-7\pi+7\operatorname{arccos}%
\left(  \frac{1}{3}\right)  \right)  ,\\
b_{\mathcal{LEHR}}=\frac{-\sqrt{2}}{9k^{2}}, & b_{\mathcal{VEHR}}=2^{\frac
{-5}{6}}3^{\frac{-2}{3}}k^{-2}\left(  -4\sqrt{2}+46\pi-46\operatorname{arccos}%
\left(  \frac{1}{3}\right)  \right)  ,\\
e_{\mathcal{LEHR}}=\frac{-\sqrt{2}}{3k^{2}}, & e_{\mathcal{VEHR}}=2^{\frac
{1}{6}}3^{\frac{-2}{3}}k^{-2}\left(  -6\sqrt{2}+5\pi-5\operatorname{arccos}%
\left(  \frac{1}{3}\right)  \right)  .
\end{array}
\]

\end{proof}

That the equal length metric is a local minimum of $\mathcal{VEHR}$
(Theorem~\ref{Combined eigen decomposition THM}.4) follows as a simple
corollary, since the nonzero eigenvalues of $\mathcal{VEHR}$ are both
positive. Also, note that the eigenspaces of the $\mathcal{LEHR}$ and
$\mathcal{VEHR}$ functionals are the same.

We will now show that, on the double tetrahedron, the only critical points of
the $\mathcal{LEHR}$ functional are the equal length metrics. Thus the
$\mathcal{L}$-Einstein metrics are unique. The fact that these critical points
are saddle points follows directly from the eigenvalues of the Hessian.

\begin{proof}
[Proof of Theorem~\ref{Combined eigen decomposition THM}.3.]By Proposition
\ref{length derivative of EHR PROP}, a critical point of $\mathcal{LEHR}$ on
the double tetrahedron satisfies%
\[
2\pi-2\beta_{ij}=\lambda\text{ for all }\left\{  i,j\right\}  \in E,
\]
where $\lambda=\frac{\mathcal{EHR}}{\mathcal{L}}$. Thus,
\begin{equation}
\beta_{ij}=\frac{2\pi-\lambda}{2}\text{ for all }\left\{  i,j\right\}  \in E.
\label{dihedral equals constant EQ}%
\end{equation}
Equation (\ref{dihedral equals constant EQ}) implies that all dihedral angles
are equal. Since all dihedral angles are equal, the face angles will
necessarily all be equal, since the spherical cosine law shows that the face
angles are determined by the dihedral angles. The faces are thus all
equilateral, and hence the generating tetrahedron has all edge lengths equal.
Therefore, the critical points of $\mathcal{LEHR}$ occur only at equal length metrics.
\end{proof}

\begin{remark}
An immediate consequence of Theorems~\ref{Combined eigen decomposition THM}.2
and \ref{Combined eigen decomposition THM}.3 is that the global extrema of the
$\mathcal{LEHR}$ functional occur on the boundary of $\mathfrak{met}%
_{\mathcal{L}}$, i.e.~on degenerate (zero volume) length structures.
\end{remark}

Despite the fact that the equal length metrics are local minima of the
$\mathcal{VEHR}$ functional, we cannot conclude that they are global minima
since the Hessian of the $\mathcal{VEHR}$ functional is not globally positive semidefinite.

\begin{proposition}
\label{one parameter family of tetrahedra PROP}Consider the\ one-parameter
family of (admissible) length structures given by (\ref{length variation}) for
$t\in\left[  1,\sqrt{2}\right)  $. There is a constant $t_{\ast}$
($\approx1.26836$) such that for $t<t_{\ast},$ the Hessian of $\mathcal{VEHR}$
is positive semidefinite with a one-dimensional nullspace, and for $t>t_{\ast
},$ the Hessian has a mixed signature. For $t=t_{\ast},$ the nullspace is
two-dimensional consisting of a scaling direction and an additional
eigenvector $v=\left(  0,1,-1,-1,1,0\right)  .$
\end{proposition}

\begin{proof}
Using the computation of the eigenvalues in Lemma \ref{hessian lemma} and
Theorem~\ref{Combined eigen decomposition THM}.2, one sees that the eigenvalue
associated to the eigenvector $v=\left(  0,1,-1,-1,1,0\right)  $ can be
continuously expressed as%
\[
\lambda_{v}^{\mathcal{VEHR}}\left(  t\right)  =-4c_{\mathcal{F}}\left(
t\right)  -2ta_{\mathcal{F}}\left(  t\right)  ,
\]
with the following value at $t=1$:%
\begin{align*}
\lambda_{v}^{\mathcal{VEHR}}(1)  &  =2^{\frac{7}{6}}3^{\frac{1}{3}}\left(
7\pi-2^{\frac{3}{2}}-7\arccos\left(  \frac{1}{3}\right)  \right)  ,\\
&  \approx34.145.
\end{align*}

However,
\[
\lambda_{v}^{\mathcal{VEHR}}\left(  1.3\right)  \approx-5.97897,
\]
thus by continuity there exists a value $t_{\ast}$ such that $\lambda
_{v}^{\mathcal{VEHR}}\left(  t_{\ast}\right)  =0$. Using Newton's method
$t_{\ast}$ can be approximated as%
\[
t_{\ast}\approx1.26836.
\]

\end{proof}

In addition, one might ask if the functional $\mathcal{VEHR}$ stays bounded on
$\mathfrak{met}_{\mathcal{V}}.$ It does not.

\begin{proposition}
\label{NEHR unbounded PROP}The $\mathcal{VEHR}$ functional is unbounded on the
double tetrahedron.
\end{proposition}

\begin{proof}
This follows from the more general Proposition \ref{VEHR unbdd in const curv}.
\end{proof}

\section{Conformal variations\label{conformal section}}

In this section, we describe the behavior of the Einstein-Hilbert-Regge
functional within a conformal class. We will describe the general set-up, and
then we will specialize to the setting of the double tetrahedron.

\subsection{Introduction to discrete conformal structures}

We will consider a certain conformal structure that has been studied in
\cite{G3}, \cite{Luo1}, \cite{RW}. For the following, let $V^{\ast}$ denote
the real-valued functions on the vertices.

\begin{definition}
Let $\{L_{e}\}_{e\in E}$ be such that $(M,\mathcal{T},L)$ is a piecewise flat
manifold. Let $U\subset V^{\ast}$ be an open set. A \emph{conformal structure}
is a map $U\rightarrow\mathfrak{met}\left(  M,\mathcal{T}\right)  $ determined
by
\begin{equation}
\ell_{e}\left(  f\right)  =\exp\left[  \frac{1}{2}\left(  f_{v}+f_{v^{\prime}%
}\right)  \right]  L_{e}, \label{pb conf}%
\end{equation}
where $e$ is the edge between $v$ and $v^{\prime}.$ The \emph{conformal class}
is the image of $U$ in $\mathfrak{met}\left(  M,\mathcal{T}\right)  ,$ and it
is entirely determined by $L.$ A \emph{conformal variation} $f\left(
t\right)  $ is a smooth curve $\left(  -\varepsilon,\varepsilon\right)
\rightarrow U$ for small $\varepsilon>0$, and it induces a \emph{conformal
variation of metrics} $\ell\left(  f\left(  t\right)  \right)  .$
\end{definition}

\begin{remark}
There is a more general notion of conformal structure on piecewise flat
manifolds that is described in \cite{G3}. The conformal structure described
here is called the perpendicular bisector conformal structure in that paper.
\end{remark}

A useful fact about the conformal structure is that the length cross ratio is
a conformal invariant.

\begin{proposition}
[\cite{SSP}]\label{conformal invariant} For a fixed conformal class and for
any tetrahedron $\{1234\}$ in the triangulation, there exist constants
$c_{13}$ and $c_{14}$ depending only on the conformal class such that
\begin{align*}
\frac{\ell_{12}\ell_{34}}{\ell_{14}\ell_{23}}  &  =c_{13},\\
\frac{\ell_{12}\ell_{34}}{\ell_{13}\ell_{24}}  &  =c_{14}.
\end{align*}
In particular, $c_{13}=\frac{L_{12}L_{34}}{L_{14}L_{23}}$, and $c_{14}%
=\frac{L_{12}L_{34}}{L_{13}L_{24}}.$
\end{proposition}

\begin{proof}
If one considers the length cross ratio, $\frac{\ell_{12}\ell_{34}}{\ell
_{14}\ell_{23}}$, one can easily see that the conformal factors cancel, and we
are left with $\frac{L_{12}L_{34}}{L_{14}L_{23}}$ which depends only on the
choice of conformal class. The same idea works for the other cross ratio.
\end{proof}

The conformal structure gives rise to certain geometric structures within each
tetrahedron (or, more precisely, within the tangent space of each
tetrahedron). Each face $f$ has a circumcenter $c_{f}$ within its tangent
plane. Each tetrahedron $t$ has a circumcenter $c_{t}$. The segment from
$c_{t}$ to $c_{f}$ (in the tangent space of $t$) is orthogonal to the tangent
plane of the face $f,$ and this segment has a signed height $h_{f}.$ Within
the tangent plane to each face, the segment from $c_{f}$ to the midpoint of
one of its edges is orthogonal to that edge, and it has a signed height
$h_{e}.$ Both $h_{f}$ and $h_{e}$ can be computed explicitly (see \cite{G1}).

We can use the results from \S 3 to show the following theorem.

\begin{theorem}
At an equal length metric, both the $\mathcal{LEHR}$ and $\mathcal{VEHR}$
functionals are convex within a conformal class.
\end{theorem}

\begin{proof}
We need only show that the conformal variations at an equal length metric lie
in the nonnegative eigenspaces. Note that the conformal directions are spanned
by
\[
\frac{\partial}{\partial f_{1}}=\frac{\ell_{12}}{2}\frac{\partial}%
{\partial\ell_{12}}+\frac{\ell_{13}}{2}\frac{\partial}{\partial\ell_{13}%
}+\frac{\ell_{14}}{2}\frac{\partial}{\partial\ell_{14}},
\]
etc. At equal length metrics, this corresponds to vectors $\left(
1,1,1,0,0,0\right)  ,$ $\left(  1,0,0,1,1,0\right)  ,$ $\left(
0,1,0,1,0,1\right)  ,$ and $\left(  0,0,1,0,1,1\right)  $ in the notation of
Theorem \ref{Combined eigen decomposition THM}. One can verify that the
conformal variations lie in the span of the nonnegative eigenspaces of the
Hessian of $\mathcal{LEHR}$. The argument for $\mathcal{VEHR}$ is trivial
since the equal length metrics are local minima in the space of all metrics.
\end{proof}

\subsection{Constant Scalar Curvature Metrics\label{csc section}}

In this section, we consider the constant scalar curvature metrics (Definition
\ref{csc def}) on the double tetrahedron. We show existence of constant scalar
curvature metrics in each conformal class, and we also find interesting
geometric examples of these metrics. Additionally, we will show that constant
scalar curvature metrics in a given conformal class are not necessarily unique.

The main results in this section are the following.

\begin{theorem}
\label{csc main thm}

\begin{enumerate}
\item An equihedral metric on the double tetrahedron has constant scalar
curvature with respect to the $\mathcal{LEHR}$ functional.

\item An equihedral metric on the double tetrahedron has constant scalar
curvature with respect to the $\mathcal{VEHR}$ functional.

\item There is a unique equihedral metric up to scaling in every conformal
class on the double tetrahedron.

\item With respect to the $\mathcal{LEHR}$ functional, constant scalar
curvature metrics are not necessarily unique in a conformal class on the
double tetrahedron.
\end{enumerate}
\end{theorem}

First, we will define a notion of vertex curvature which will lead to a
definition of constant scalar curvature metrics.

\begin{definition}
The \emph{vertex curvature} of a vertex $v$ is%
\begin{equation}
K_{v}=\frac{1}{2}\sum_{e>v}K_{e}, \label{vertex curvature}%
\end{equation}
where the sum is over all edges containing vertex $v.$
\end{definition}

\begin{remark}
$\mathcal{EHR}$ can be written in terms of vertex curvatures as
\[
\mathcal{EHR}\left(  M,\mathcal{T},\ell\right)  =\sum_{v\in V}K_{v}.
\]

\end{remark}

\begin{remark}
As described in \cite{G3}, the formula for vertex curvature depends on the
choice of conformal structure, which comes from the variation formula for
$\mathcal{EHR}$.
\end{remark}

One can immediately see the connection between discrete conformal variations
and those in the smooth category when one considers the following variation formulas.

\begin{lemma}
For a conformal variation $f(t)$ of a three-dimensional, piecewise flat
manifold $(M,\mathcal{T},\ell)$, we have
\begin{subequations}
\label{var of ehr}%
\begin{align}
&  \frac{\partial}{\partial f_{v}}\mathcal{EHR}(M,\mathcal{T},\ell
(f))=K_{v},\\
&  \frac{\partial}{\partial f_{v}}\mathcal{LEHR}(M,\mathcal{T},\ell
(f))=\mathcal{L}^{-1}(K_{v}-(\mathcal{LEHR})L_{v}),\\
&  \frac{\partial}{\partial f_{v}}\mathcal{VEHR}(M,\mathcal{T},\ell
(f))=\mathcal{V}^{-\frac{1}{3}}(K_{v}-\frac{\mathcal{EHR}}{3\mathcal{V}}%
V_{v}),
\end{align}
where $L_{v}=\frac{1}{2}\sum_{e>v}\ell_{e},$ $V_{v}=\frac{1}{3}\sum
_{t>f>v}h_{f<t}A_{f}$, and $A_{f}$ is the area of face $f$.
\end{subequations}
\end{lemma}

\begin{proof}
These formulas are mostly in \cite{G3}, and the rest follow easily.
\end{proof}

\begin{remark}
In the previous formulas, $L_{v}=\frac{\partial\mathcal{L}}{\partial f_{v}}$
and $V_{v}=\frac{\partial\mathcal{V}}{\partial f_{v}}$. Also note that
$\sum_{v\in V}L_{v}=\mathcal{L}$, and $\sum_{v\in V}V_{v}=3\mathcal{V}.$
\end{remark}

These variation formulas motivate the following notion of constant scalar
curvature in the discrete setting.

\begin{definition}
\label{csc def}A three-dimensional piecewise flat manifold $(M,\mathcal{T}%
,\ell)$ has \emph{constant scalar curvature} if it is a critical point of one
of the normalized $\mathcal{EHR}$ functionals with respect to conformal
variations. We say $(M,\mathcal{T},\ell)$ \emph{has constant $\mathcal{L}%
$-scalar curvature $\lambda_{\mathcal{L}}$} if
\begin{equation}
K_{v}=\lambda_{\mathcal{L}}L_{v}, \label{Lcsc eqn}%
\end{equation}
for all $v\in V$. Here $\lambda_{\mathcal{L}}=\mathcal{LEHR}$. We say
$(M,\mathcal{T},\ell)$ has \emph{constant $\mathcal{V}$-scalar curvature
$\lambda_{\mathcal{V}}$} if
\begin{equation}
K_{v}=\lambda_{\mathcal{V}}V_{v}, \label{Vcsc eqn}%
\end{equation}
for all $v\in V$. Here $\lambda_{\mathcal{V}}=\frac{\mathcal{EHR}%
}{3\mathcal{V}}$.
\end{definition}

The following proposition is essentially shown in \cite{G3}.

\begin{proposition}
If $(M,\mathcal{T},\ell)$ is a three-dimensional piecewise flat manifold which
is $\mathcal{L}$-Einstein ($\mathcal{V}$-Einstein), then it has constant
$\mathcal{L}$-scalar curvature ($\mathcal{V}$-scalar curvature).
\end{proposition}

Now we will define the notion of an equihedral metric. We will also collect
the necessary pieces to prove that the equihedral metrics have both constant
$\mathcal{L}$-scalar curvature and constant $\mathcal{V}$-scalar curvature.

\begin{definition}
Let $\sigma^{3}$ be a tetrahedron such that all faces are congruent. Then
$\sigma^{3}$ is called an \emph{equihedral tetrahedron}.
\end{definition}

\begin{theorem}
\label{equi def} Let $\sigma^{3}$ be a tetrahedron. One can show that the
following conditions are equivalent:

\begin{enumerate}
\item The tetrahedron $\sigma^{3}$ is equihedral.

\item The opposite edges have equal lengths.

\item The opposite edges have equal dihedral angles.

\item The inscribed sphere touches each face at its circumcenter.

\item The center of the inscribed sphere corresponds to the circumcenter of
the tetrahedron.
\end{enumerate}
\end{theorem}

In fact, according to \cite{Ar}, there are over 100 equivalent conditions that
characterize equihedral tetrahedra. We note that in the literature, these
tetrahedra are sometimes called \textquotedblleft isosceles
tetrahedra\textquotedblright\ or \textquotedblleft equifacial
tetrahedra.\textquotedblright\ We will say that a piecewise flat metric $\ell$
on the double tetrahedron is \emph{equihedral} if its generating tetrahedron
satisfies any of the equivalent conditions in Theorem \ref{equi def}.

\begin{lemma}
\label{l sums equal} In an equihedral tetrahedron, $L_{v}=L_{v^{\prime}}\ $for
all pairs of vertices $v,v^{\prime}\in V$.
\end{lemma}

\begin{proof}
Consider, for example, $L_{1}=\frac{1}{2}(\ell_{12}+\ell_{13}+\ell_{14})$.
Using Theorem~\ref{equi def}.2, one sees that this sum is equal to, say,
$\frac{1}{2}(\ell_{21}+\ell_{24}+\ell_{23})=L_{2}$. A similar argument holds
for the other vertices.
\end{proof}

\begin{lemma}
\label{equihedral h and V equal}In an equihedral tetrahedron, the
$h_{f}=h_{f^{\prime}}$ for all faces $f,f^{\prime}\in F$, and hence the
$V_{v}=V_{v^{\prime}}$ for all vertices $v,v^{\prime}\in V$.
\end{lemma}

\begin{proof}
The definition of equihedral tetrahedron implies that the $A_{f}$ are equal
for all faces. Since the geometric center $c_{f}$, corresponds to the
circumcenter of $f$, Theorem~\ref{equi def}.4 and \ref{equi def}.5 combine to
show that the $h_{f}$ are all equal. Since, $V_{v}=\frac{1}{3}\sum
_{t>f>v}h_{f<t}A_{f}$, the $V_{v}$ are equal as well.
\end{proof}

We will now show that equihedral metrics have constant scalar curvature in the
sense of both (\ref{Lcsc eqn}) and (\ref{Vcsc eqn}).

\begin{proof}
[Proof of Theorem~\ref{csc main thm}.1]Let $\ell$ be an equihedral metric on
the double tetrahedron. Label the generating tetrahedron $1234$. We would like
to show that $\ell$ has constant $\mathcal{L}$-scalar curvature $\lambda
_{\mathcal{L}}$; i.e., that $\ell$ satisfies (\ref{Lcsc eqn}). By
Theorem~\ref{equi def}.2 and \ref{equi def}.3, an equihedral tetrahedron has
opposite edge lengths and opposite dihedral angles equal. This implies that
$\ell_{12}=\ell_{34}$, $\ell_{13}=\ell_{24}$, $\ell_{14}=\ell_{23}$,
$\beta_{12}=\beta_{34}$, $\beta_{13}=\beta_{24}$, and $\beta_{14}=\beta_{23}$.
One can easily check that this implies that the $K_{v}$ are equal for all
$v\in V$; hence $4K_{v}=\mathcal{EHR}$. By Lemma \ref{l sums equal},
$4L_{v}=\mathcal{L}$. Thus, $\frac{K_{v}}{L_{v}}=\mathcal{LEHR}$ as required.
\end{proof}

\begin{proof}
[Proof of Theorem~\ref{csc main thm}.2]Let $\ell$ be an equihedral metric on
the double tetrahedron. We would like to show that $\ell$ has constant
$\mathcal{V}$-scalar curvature $\lambda_{\mathcal{V}}$; i.e. that $\ell$
satisfies (\ref{Vcsc eqn}).

As above, all of the $K_{v}$ are equal. Then for each $v$, $K_{v}=\frac{1}%
{4}\mathcal{EHR}$. Since all of the $V_{v}$ are equal by Lemma
\ref{equihedral h and V equal}, we have that $V_{v}=\frac{3}{4}\mathcal{V}$
for every $v\in V$. Then the ratio $\frac{K_{v}}{V_{v}}=\frac{\mathcal{EHR}%
}{3\mathcal{V}}$ as required.
\end{proof}

We now show that there is a one-parameter family of constant scalar curvature
metrics in every conformal class by showing the existence of equihedral
metrics in every conformal class.

\begin{proof}
[Proof of Theorem~\ref{csc main thm}.3]We begin by fixing a conformal class,
$\{L_{ij}\}$. We would like to solve the following system of equations
$\ell_{ij}=\ell_{k\ell}$, $\ell_{ik}=\ell_{j\ell}$, and $\ell_{i\ell}%
=\ell_{jk}$ in our given conformal class. Using the conformal invariants from
Proposition~\ref{conformal invariant}, one obtains%

\begin{align*}
e^{\frac{f_{i}}{2}}  &  =\sqrt{\frac{L_{j\ell}L_{k\ell}}{L_{ij}L_{ik}}%
}e^{\frac{f_{\ell}}{2}}\\
e^{\frac{f_{j}}{2}}  &  =\sqrt{\frac{L_{i\ell}L_{k\ell}}{L_{jk}L_{ij}}%
}e^{\frac{f_{\ell}}{2}}\\
e^{\frac{f_{k}}{2}}  &  =\sqrt{\frac{L_{i\ell}L_{j\ell}}{L_{ik}L_{jk}}%
}e^{\frac{f_{\ell}}{2}}.
\end{align*}
We see that, modulo scaling, there is a unique equihedral metric in every
conformal class as required.
\end{proof}

A natural question is the uniqueness of constant scalar metrics in a conformal
class. On the double tetrahedron, constant $\mathcal{L}$-scalar curvature
metrics are not necessarily unique within a conformal class. Notice that the
previous proof implies that such additional constant $\mathcal{L}$-scalar
curvature metrics are not equihedral.

\begin{proof}
[Proof of Theorem~\ref{csc main thm}.4]There can be multiple constant
$\mathcal{L}$-scalar curvature metrics within a fixed conformal class. As a
specific example of this, consider the conformal class given by%
\[
L_{e}=1\text{, for all }e\in E.
\]
Constant scalar curvature metrics occur for the following conformal
parameters:%
\begin{align*}
F_{a}  &  =\left(  f_{1},f_{2},f_{3},f_{4}\right)  =\left(  0,0,0,0\right)
,\\
F_{b}  &  =\left(  f_{1},f_{2},f_{3},f_{4}\right)  \approx\left(
-1.233,-1.233,0,0\right)  .
\end{align*}
Notice that the equal length metric is given by $F_{a}$, and another
non-equihedral metric is given by $F_{b}$. We conjecture that non-equihedral
constant scalar curvature metrics exist in every conformal class.
\end{proof}

\subsection{Convexity of $\mathcal{LEHR}$}

In this section, we consider the convexity of the $\mathcal{LEHR}$ functional
within a fixed conformal class. We have previously shown that at an equal
length metric, the $\mathcal{LEHR}$ functional restricted to a conformal class
is convex. Direct computation of the Hessian allows us to analyze the
$\mathcal{LEHR}$ functional at equihedral metrics. The main results of this
section are the following.

\begin{theorem}
\label{lehr main thm}

\begin{enumerate}
\item The discrete Laplacian is negative semidefinite at equihedral metrics.

\item The $\mathcal{LEHR}$ functional is convex in a conformal class at some,
but not all, equihedral metrics.
\end{enumerate}
\end{theorem}

In order to analyze the $\mathcal{LEHR}$ functional in a conformal class, we
will compute the Hessian at a constant scalar curvature metric. First we need
the definition of dual length.

\begin{definition}
Let $(M,\mathcal{T},\ell)$ be a three-dimensional piecewise flat manifold. The
\emph{dual length} $\ell_{e}^{\ast}$ is defined as
\begin{equation}
\ell_{e}^{\ast}=\sum_{t>e}\frac{1}{2}(h_{e<f}h_{f<t}+h_{e<f^{\prime}%
}h_{f^{\prime}<t}), \label{dual length}%
\end{equation}
where $f$ and $f^{\prime}$ are the faces of $t$ containing $e.$
\end{definition}

\begin{remark}
Although we call $\ell_{e}^{\ast}$ a \textquotedblleft dual
length,\textquotedblright\ it is actually a (signed) area of a (dual) cell
orthogonal to the edge $e$. For further details see \cite{G1}.
\end{remark}

Now we will compute the Hessian of the $\mathcal{LEHR}$ functional. Since we
will only analyze the Hessian at constant scalar curvature metrics, we include
only that formula. The more general formula follows in a straightforward
fashion from the results presented in this paper and in \cite{G1}.

\begin{proposition}
The Hessian of the $\mathcal{LEHR}$ functional at a constant $\mathcal{L}%
$-scalar curvature metric is
\[
\frac{\partial^{2}\mathcal{LEHR}}{\partial f_{v} \partial f_{v^{\prime}}%
}=-\frac{2}{\mathcal{L}}\Delta_{vv^{\prime}}+\frac{1}{\mathcal{L}%
}N_{vv^{\prime}},
\]
where
\[
\triangle_{vv^{\prime}}=\left\{
\begin{array}
[c]{cc}%
\frac{\ell_{e}^{\ast}}{\ell_{e}} & \text{if }e=vv^{\prime},\\
-\sum_{e>v}\frac{\ell_{e}^{\ast}}{\ell_{e}} & \text{if }v=v^{\prime},\\
0 & \text{otherwise}%
\end{array}
\right.
\]
is the Laplacian matrix, and
\[
N_{vv^{\prime}}=\left\{
\begin{array}
[c]{cc}%
\frac{1}{4}(K_{e}-(\mathcal{LEHR})\ell_{e}) & \text{if }e=vv^{\prime},\\
\frac{1}{2}(K_{v} - (\mathcal{LEHR})L_{v}) & \text{if }v=v^{\prime},\\
0 & \text{otherwise.}%
\end{array}
\right.
\]

\end{proposition}

\begin{corollary}
At an equihedral metric,
\[
N_{vv^{\prime}}=\left\{
\begin{array}
[c]{cc}%
\frac{1}{4}\left(  K_{e}-\frac{K_{v}\ell_{e}}{L_{v}}\right)  & \text{if
}e=vv^{\prime}\\
0 & \text{otherwise.}%
\end{array}
\right.
\]

\end{corollary}

\begin{proof}
We have shown in the proofs of Theorem~\ref{csc main thm}.2 and
Lemma~\ref{l sums equal} that at an equihedral metric, the vertex curvatures
and the $L_{v}$ are constant for all $v\in V$. Then $K_{v}=\frac{1}%
{4}\mathcal{EHR}$, and $L_{v}=\frac{1}{4}\mathcal{L}$. So at an equihedral
metric, $(\mathcal{LEHR})\ell_{e}=-\frac{K_{v}\ell_{e}}{L_{v}}$ if
$e=vv^{\prime}$, and $(\mathcal{LEHR})L_{v}=K_{v}$ if $v=v^{\prime}$. These
formulas give the desired result.
\end{proof}

We would like to analyze the convexity of the $\mathcal{LEHR}$ functional at
equihedral metrics. First, we will show that the Laplacian matrix $\triangle$
is negative semidefinite at equihedral metrics.

Recall that a matrix, $A$, is said to be \emph{diagonally dominant} if
$|a_{ii}|\geq\sum_{j\neq i}|a_{ij}|$, where $a_{ij}$ denotes the entry in the
$i^{th}$ row and the $j^{th}$ column. The following theorem is well-known.

\begin{theorem}
\label{diag dom} A Hermitian, diagonally dominant matrix with real nonnegative
diagonal entries is positive semidefinite.
\end{theorem}

\begin{proposition}
\label{laplacian nsd} The Laplacian $\triangle$ is negative semidefinite if
the dual lengths $\ell_{e}^{\ast}$ are nonnegative for a given metric.
\end{proposition}

\begin{proof}
Notice that if $\ell_{e}^{\ast}\geq0,$ then $\triangle$ is diagonally
dominant, and so by Theorem \ref{diag dom}, $\triangle$ is negative semidefinite.
\end{proof}

\begin{remark}
The set of metrics for which $\triangle$ is negative semidefinite forms an
open set containing the equal length metrics, as does the set of metrics for
which $\ell_{e}^{\ast}$ are positive. It is possible that our analogies for
Riemannian geometry depend on one of these two conditions, and that other
metrics correspond to some sort of more general infinitesimal geometry such as
sub-Riemannian geometry.
\end{remark}

We will now show that the Laplacian is negative semidefinite at equihedral metrics.

\begin{lemma}
\label{positive edge height}In an equihedral tetrahedron, all faces are acute
triangles. Thus $h_{e<f}\geq0$ for each edge $e$ and face $f.$
\end{lemma}

\begin{proof}
The proof that the faces of an equihedral tetrahedron are acute triangles
follows that suggested in \cite{KVW}. Let $1234$ be an equihedral tetrahedron
such that $\ell_{12}=\ell_{34}=a$, $\ell_{13}=\ell_{24}=b$, and $\ell
_{14}=\ell_{23}=c$. Flatten the tetrahedron by fixing $a$ and $c$; allow
$\ell_{24}$ to grow until faces $123$ and $134$ become coplanar, say when
$\ell_{24}=d$. One obtains the parallelogram $1234$ with sides of lengths $a$
and $c$ and with diagonals of lengths $b$ and $d$. Since $d>b$, the law of
cosines implies that the angle at vertex $2<123$ is acute. A similar argument
can be used to show all face angles are acute.

It is well-known that the circumcenter of an acute triangle lies in the
interior of the triangle, so $h_{e<f}\geq0$ as required.
\end{proof}

\begin{lemma}
\label{dual length pos} For any equihedral metric on $\operatorname{DT},$
$\ell_{e}^{\ast}\geq0$ for every edge.
\end{lemma}

\begin{proof}
Since $\sum_{f\in F}h_{f}A_{f}=3\mathcal{V}$, both $\mathcal{V}$ and $A_{f}$
are positive, and $h_{f}$ is the same for each face by
Lemma~\ref{equihedral h and V equal}, we must have that $h_{f}\geq0$ for all
$f\in F.$ It follows from Proposition \ref{positive edge height} that
$h_{e<f}\geq0.$ Since $\ell_{e}^{\ast}=h_{e<f}h_{f}+h_{e<f^{\prime}%
}h_{f^{\prime}},$ the result follows.
\end{proof}

\begin{proof}
[Proof of Theorem~\ref{lehr main thm}.1]That the Laplacian is negative
semidefinite at equihedral metrics follows directly from
Proposition~\ref{laplacian nsd} and Lemma~\ref{dual length pos}.
\end{proof}

We will show below that many, but not all, equihedral metrics are local minima
of $\mathcal{LEHR}$. Note that we are still unable to find which equihedral
metrics are local minima of $\mathcal{VEHR}$ due to the influence of the
normalization factor on the Hessian.

\begin{proof}
[Proof of Theorem~\ref{lehr main thm}.2]We will show that the $\mathcal{LEHR}$
functional is convex at some but not all equihedral metrics. Consider the one
parameter family of length vectors (\ref{length variation}), which is a
variation through equihedral metrics. Note that along this family, each vertex
looks the same, and so we can write the Hessian of $\mathcal{LEHR}$ restricted
to the conformal class as
\[
\left.  Hess\left(  \mathcal{LEHR}\right)  \right\vert _{\ell\left(  t\right)
}=%
\begin{bmatrix}
b & c & a & a\\
c & b & a & a\\
a & a & b & c\\
a & a & c & b
\end{bmatrix}
\]
where
\[
a=\frac{\partial^{2}\mathcal{LEHR}}{\partial f_{1}\partial f_{3}%
},\;\;\;b=\frac{\partial^{2}\mathcal{LEHR}}{\partial f_{1}^{2}},\;\;\;c=\frac
{\partial^{2}\mathcal{LEHR}}{\partial f_{1}\partial f_{2}},
\]
and we always calculate at $f_{1}=f_{2}=f_{3}=f_{4}=0.$ Furthermore, we know
that $\mathcal{LEHR}$ is invariant under uniform scaling, so $\left(
1,1,1,1\right)  $ is a zero eigenvalue, which implies that $2a+b+c=0.$ One can
now check that the eigenspaces decompose as follows:
\[%
\begin{tabular}
[c]{|l|c|c|}\hline
eigenspace & spanning vectors & eigenvalues\\\hline
$V_{\lambda_{1}}$ & $%
\begin{array}
[c]{c}%
\left(  1,-1,0,0\right) \\
\left(  0,0,1,-1\right)
\end{array}
$ & $-2a-2c,$\\\hline
$V_{\lambda_{2}}$ & $%
\begin{array}
[c]{c}%
\left(  1,1,-1,-1\right)
\end{array}
$ & $-4a,$\\\hline
$V_{0}$ & $%
\begin{array}
[c]{c}%
\left(  1,1,1,1\right)
\end{array}
$ & $0.$\\\hline
\end{tabular}
\ \ \ \
\]

One can show that at the equal length metric $\ell_{1}$ where all lengths
equal $1$, the eigenvalues of the Hessian of $\mathcal{LEHR}$ are $\left(
\frac{4\sqrt{2}}{9},\frac{4\sqrt{2}}{9},\frac{4\sqrt{2}}{9},0\right)  $. Since
the eigenvalues are continuous and the zero eigenvalue persists, the Hessian
of $\mathcal{LEHR}$ is positive semidefinite in a neighborhood of $\ell_{1}.$
For $t=1.35$, the eigenvalue $\lambda_{2}\approx-0.238$, so by continuity of
the eigenvalues, it must be zero for some $t$. A computer algebra package
allows us to approximate this value to be $t\approx1.31471$.
\end{proof}

\section{The Yamabe invariant\label{general section}}

As in the smooth case, one can consider invariants based on the
Einstein-Hilbert functional. In the smooth setting, one considers the Yamabe
constant of a conformal class of Riemannian metrics. Recall that if $g$ is a
Riemannian metric on a manifold $M$, the conformal class $\left[  g\right]  $
is defined by
\[
\left[  g\right]  =\left\{  e^{f}g:f\in C^{\infty}\left(  M\right)  \right\}
.
\]
One then considers the \emph{Yamabe constant} $Y\left(  M,\left[  g\right]
\right)  $ to be
\begin{equation}
Y\left(  M,\left[  g\right]  \right)  =\inf\left\{  \mathcal{NEH}\left(
M,g_{0}\right)  :g_{0}\in\left[  g\right]  \right\}  .
\end{equation}
One can show (see, e.g., \cite{SY}) that for any $M^{n}$ and $g,$
\[
Y\left(  M^{n},\left[  g\right]  \right)  \leq Y\left(  S^{n},\left[
g_{\operatorname{can}}\right]  \right)
\]
where $g_{\operatorname{can}}$ is any round metric, and so it makes sense to
consider the \emph{Yamabe invariant} for $M$ (also called the sigma constant),
namely%
\begin{equation}
Y\left(  M^{n}\right)  =\sup Y\left(  M^{n},\left[  g\right]  \right)  ,
\label{yamabe smooth}%
\end{equation}
where the sup is over all conformal classes. Note that the Yamabe invariant
has been computed for a number of manifolds, including the sphere, real
projective space, sphere bundles over the circle, and hyperbolic manifolds
(\cite{AN}, \cite{And2}, \cite{BN}, \cite{Kob}, \cite{Sch2}).

One may consider similar questions for the Einstein-Hilbert-Regge functional.
We now have two formulations of Yamabe constants based on different normalizations:

\begin{definition}
Let $\left[  \ell\right]  $ denote the set of all metrics which are conformal
to $\ell$, i.e.,
\[
\left[  \ell\right]  =\left\{  \tilde{\ell}\in\mathfrak{met}:\tilde{\ell
}\left(  e\right)  =\exp\left[  \frac{1}{2}\left(  f\left(  v\right)
+f\left(  v^{\prime}\right)  \right)  \right]  \ell\left(  e\right)  \text{
for all }e=vv^{\prime}\right\}  .
\]
The $\mathcal{L}$\emph{-Yamabe constant} $Y_{\mathcal{L}}\left(
M,\mathcal{T},\left[  \ell\right]  \right)  $ is defined to be
\[
Y_{\mathcal{L}}\left(  M,\mathcal{T},\left[  \ell\right]  \right)
=\inf\left\{  \mathcal{LEHR}\left(  M,\mathcal{T},\tilde{\ell}\right)
:\tilde{\ell}\in\left[  \ell\right]  \right\}  .
\]
The $\mathcal{V}$\emph{-Yamabe constant} $Y_{\mathcal{V}}\left(
M,\mathcal{T},\left[  \ell\right]  \right)  $ is defined to be
\[
Y_{\mathcal{V}}\left(  M,\mathcal{T},\left[  \ell\right]  \right)
=\inf\left\{  \mathcal{VEHR}\left(  M,\mathcal{T},\tilde{\ell}\right)
:\tilde{\ell}\in\left[  \ell\right]  \right\}  .
\]
The $\mathcal{L}$\emph{-Yamabe invariant} $Y_{\mathcal{L}}\left(
M,\mathcal{T}\right)  $ is defined to be
\[
Y_{\mathcal{L}}\left(  M,\mathcal{T}\right)  =\sup\left\{  Y_{\mathcal{L}%
}\left(  M,\mathcal{T},\left[  \ell\right]  \right)  :\ell\in\mathfrak{met}%
\right\}  .
\]
The $\mathcal{V}$\emph{-Yamabe invariant} $Y_{\mathcal{V}}\left(
M,\mathcal{T}\right)  $ is defined to be
\[
Y_{\mathcal{V}}\left(  M,\mathcal{T}\right)  =\sup\left\{  Y_{\mathcal{V}%
}\left(  M,\mathcal{T},\left[  \ell\right]  \right)  :\ell\in\mathfrak{met}%
\right\}  .
\]

\end{definition}

The $\mathcal{L}$-Yamabe invariant seems well-motivated on the double
tetrahedron, since Theorem \ref{Combined eigen decomposition THM} indicates
that at the unique $\mathcal{L}$-Einstein metric, with regard to the Hessian
of the $\mathcal{LEHR}$, conformal directions span the eigenspaces
corresponding to nonnegative eigenvalues and the orthogonal space is spanned
by eigenspaces corresponding to negative eigenvalues. The definition is less
well-motivated for the $\mathcal{V}$-Yamabe invariant since the equal length
metrics on $\operatorname{DT}$ are local minima for $\mathcal{VEHR}$. We still
define the Yamabe invariant in analogy because it is known that for
sufficiently fine triangulations, $\mathcal{EHR}$ and $\mathcal{V}$ converge
to $\mathcal{EH}$ and volume (see \cite{CMS}), and so for certain
triangulations one would expect the definition to be meaningful.

There is additional motivation for these definitions in slightly different
contexts. In particular, there are the following results for the corresponding
boundary value problem where the boundary geometry is fixed. In \cite{BI},
Bobenko and Izmestiev show that for a convex polyhedron with one interior
vertex, the Hessian of $\mathcal{EHR}$ has only one positive eigenvector with
a one-dimensional eigenspace. Also, in \cite{IS}, Izmestiev and Schlenker show
that for any convex polyhedron, the eigenspace corresponding to positive
eigenvectors has dimension at least as large as the number of interior
vertices. These two results indicate that $\mathcal{EHR}$ tends to be convex
within a conformal class.

We now consider the well-posedness of computing the Yamabe invariants. It is
an easy consequence of the definition of $\mathcal{LEHR}$ that
\[
0\leq\mathcal{LEHR}\leq2\pi,
\]
and so $Y_{\mathcal{L}}\left(  \operatorname{DT}\right)  $ certainly exists.
We are not yet able to compute its value, however. In fact, although we have
found a large class of constant $\mathcal{L}$-scalar curvature metrics, we are
not even able to compute the $\mathcal{L}$-Yamabe constant for the conformal
class containing the equal length metrics.

\begin{problem}
Does $Y_{\mathcal{L}}\left(  \operatorname{DT},\left[  \ell_{k}\right]
\right)  =\mathcal{LEHR}\left(  \operatorname{DT},\ell_{k}\right)  ,$ where
$\ell_{k}$ is an equal length metric? Does $Y_{\mathcal{L}}\left(
\operatorname{DT}\right)  =\mathcal{LEHR}\left(  \operatorname{DT},\ell
_{k}\right)  $?
\end{problem}

Normalization with respect to volume is more complicated. A natural question
is whether $\mathcal{VEHR}$ is bounded in any sense. On the double
tetrahedron, it must be bounded below by $0.$ For the Yamabe invariant to be
meaningful, we would need the supremum of Yamabe constants to be bounded. One
might simply ask if $\mathcal{VEHR}$ is bounded on the set of constant scalar
curvature metrics (which are critical points to $\mathcal{VEHR}$ in a
conformal class). It turns out this is not true.

\begin{proposition}
\label{VEHR unbdd in const curv}On the double tetrahedron, the $\mathcal{VEHR}%
$ functional is not bounded above on the set of constant scalar curvature metrics.
\end{proposition}

\begin{proof}
Consider the one-parameter family of equihedral, hence constant scalar
curvature, metrics on the double tetrahedron given by (\ref{length variation}%
). As $t$ increases to $\sqrt{2}$, the generating tetrahedron becomes a flat
square with two diagonals. It follows that the volume $\mathcal{V}(t)$ goes to
zero, the dihedral angles $\beta_{12}$ and $\beta_{23}$ go to $\pi,$ and the
dihedral angles at all other edges go to zero. Thus $\mathcal{EHR}$ approaches
$8\pi.$ Therefore, we can make the $\mathcal{VEHR}$ functional as large as we
want simply by letting $t$ approach $\sqrt{2}$ from below.
\end{proof}

\noindent Note that Proposition \ref{VEHR unbdd in const curv} does not by
itself imply that the Yamabe invariant of $\operatorname{DT}$ is infinity,
since it is possible that there are smaller constant scalar curvature metrics
than the ones considered.

\begin{problem}
Is $Y_{\mathcal{V}}\left(  \operatorname{DT}\right)  =\infty$?
\end{problem}

In conclusion, we will make some comments about computation of the Yamabe
constants and Yamabe invariants for general triangulated piecewise flat
manifolds. In general, we will need to prove a lower bound for $\mathcal{LEHR}%
$ or $\mathcal{VEHR}$ within a conformal class just to compute the Yamabe
constant for that class, and then upper bounds to compute the Yamabe
invariant. There are fairly easy bounds for $\mathcal{LEHR}$, but we will need
an additional \textquotedblleft fatness\textquotedblright\ criterion for
$\mathcal{VEHR}$. Note that fatness requirements are quite common in theorems
in piecewise flat geometry (e.g., \cite{CMS}, \cite{Sto}). The most natural
fatness condition for the volume normalized functional is the following:

\begin{definition}
A piecewise flat manifold $\left(  M,\mathcal{T},\ell\right)  $ is
$\varepsilon$-fat if
\begin{equation}
\frac{\mathcal{V}}{\left(  \sum\ell_{e}\right)  ^{3}}\geq\varepsilon>0.
\label{fatness}%
\end{equation}

\end{definition}

The idea is that the volume cannot become too small without the lengths
becoming small. Notice that the fatness condition is satisfied if, for
instance, there is lower bound for the volume of each tetrahedron in terms of
its maximum edge length, though in this case $\varepsilon$ would depend on the
number of edges in $\mathcal{T}.$

The following theorem gives bounds on the $\mathcal{LEHR}$ and $\mathcal{VEHR}%
$ functionals.

\begin{proposition}
\label{proposition lower bound}Let $\left(  M,\mathcal{T},\ell\right)  $ be a
triangulated piecewise flat three-manifold and let $D_{M}$ be the maximal edge
degree (tetrahedra incident on an edge). Then

\begin{itemize}
\item For $\mathcal{LEHR}\left(  M,\mathcal{T},\ell\right)  $ we have
\[
2\pi-\pi D_{M}\leq\mathcal{LEHR}\left(  M,\mathcal{T},\ell\right)  \leq2\pi.
\]

\item If the triangulation is $\varepsilon$-fat, then there exists a constant
$C\left(  D_{M},\varepsilon\right)  $ such that
\[
\mathcal{VEHR}\left(  M,\mathcal{T},\ell\right)  \geq C
\]

\end{itemize}
\end{proposition}

\begin{proof}
For $\mathcal{LEHR}\left(  M,\mathcal{T},\ell\right)  ,$ the estimate is:
\[
2\pi-\pi D_{M}\leq\frac{\sum_{e}\left(  2\pi-\sum_{t}\beta_{e,t}\right)
\ell_{e}}{\sum_{e}\ell_{e}}\leq2\pi.
\]
For $\mathcal{VEHR}\left(  M,\mathcal{T},\ell\right)  ,$ we see that
\[
\mathcal{VEHR}\left(  M,\mathcal{T},\ell\right)  =\frac{\sum_{e}\left(
2\pi-\sum_{t}\beta_{e,t}\right)  \ell_{e}}{\mathcal{V}^{1/3}}\geq\min\left\{
0,2\pi-\pi D_{M}\right\}  \varepsilon^{-1/3}.
\]

\end{proof}

\begin{remark}
The condition of fatness in \cite{CMS} is slightly different, stating that
there is a constant $\delta>0$ such that for each $n$-simplex $\sigma^{n},$
\[
\max\left\{  \ell_{e}:e\in E\right\}  ^{-n}\left\vert \sigma^{n}\right\vert
\geq\delta,
\]
where $\left\vert \sigma^{n}\right\vert $ is the $n$-dimensional volume of
$\sigma.$ If one assumes the condition from \cite{CMS}, then
\[
\delta\leq\frac{\left\vert \sigma^{3}\right\vert }{\max\left\{  \ell_{e}:e\in
E\right\}  ^{3}}\leq\frac{\left\vert E\right\vert \mathcal{V}}{\left(
\sum\ell_{e}\right)  ^{3}},
\]
and so we can take $\varepsilon=\delta/\left\vert E\right\vert $ in
(\ref{fatness}), and we get a bound
\[
\mathcal{VEHR}\left(  M,\mathcal{T},\ell\right)  \geq\min\left\{  0,2\pi-\pi
D_{M}\right\}  \left\vert E\right\vert ^{1/3}\delta^{-1/3}.
\]

\end{remark}

Proposition \ref{proposition lower bound} justifies taking the infimum of
$\mathcal{LEHR}$ or $\mathcal{VEHR}$ over appropriate subsets of a conformal
class. Note that the estimates depends on edge degree, so that it would not be
useful if one wished to take a sequence of triangulations converging to a
smooth manifold. This motivates the following problem.

\begin{problem}
Within a discrete conformal class, is $\mathcal{LEHR}$ or $\mathcal{VEHR}$
bounded below by a constant independent of the triangulation?
\end{problem}

Note that in the smooth case, one can essentially use a reformulation of
$\mathcal{NEH}$ and a H\"{o}lder estimate to prove that such a lower bound
exists (see, e.g., \cite[Chapter V]{SY}).

Recall that for $\operatorname{DT},$ the known $\mathcal{V}$-Einstein metric
is a local minimum in $\mathfrak{met}.$ A natural question is whether there
are any negative eigenvalues at all for the Hessian of $\mathcal{VEHR}$ at a
$\mathcal{V}$-Einstein metric; in other words, is there a necessity for
maximizing to find a critical point. We have found the following:

\begin{theorem}
\label{theorem existence negative directions}There is a triangulation of the
three-sphere which admits a family of constant scalar curvature metrics such
that the maximum of $\mathcal{LEHR}$ over that family is a $\mathcal{L}%
$-Einstein metric and the maximum of $\mathcal{VEHR}$ over that family is a
$\mathcal{V}$-Einstein metric.
\end{theorem}

\begin{proof}
[Proof (sketch)]This example was first suggested to us by J. Weeks. One can
consider the $600$-cell, which is the boundary of a regular polytope in
$\mathbb{R}^{4}$. It can be given a metric such that every tetrahedron is
isometric to the other. The combinatorics are such that every vertex looks
like every other, and it is an easy consequence that such a metric must have
constant scalar curvature in either normalization. This gives a six parameter
family of constant scalar curvature metrics on the $600$-cell. One can then
check that this family has a maximum when all lengths are the same, which
gives an Einstein metric. More details will appear in \cite{CGY}.
\end{proof}

Theorem \ref{theorem existence negative directions} indicates that there are
positive and negative eigenvalues for $\mathcal{LEHR}$ and $\mathcal{VEHR}$ at
Einstein metrics. A major problem is to classify these directions.

\begin{problem}
Is there a geometric way to classify the positive, negative, and zero
eigenspaces for the Hessian of $\mathcal{LEHR}$ or $\mathcal{VEHR}$? Do we
need to restrict to certain classes of triangulations (e.g., with sufficient fatness)?
\end{problem}

\end{document}